\documentclass[11pt,a4paper]{article}

\usepackage{graphicx}
\usepackage{amsmath}
\usepackage{amsfonts}
\usepackage{amssymb}
\usepackage{enumerate}
\usepackage{lscape}
\usepackage{longtable}
\usepackage{rotating}
\usepackage{multirow}
\usepackage{color}
\usepackage{url}
\usepackage{subfigure}
\usepackage{rotating}
\newtheorem{theorem}{Theorem}[section]

\usepackage[ruled,vlined,noline,linesnumbered]{algorithm2e}

\newtheorem{lemma}{Lemma}[section]

\newtheorem{assumption}{Assumption}[section]

\newenvironment{proof}[1][Proof]{\textbf{#1.} }{\ \rule{0.5em}{0.5em} \vspace{1ex}}

\title{Decentralized Langevin Dynamics}
\author{Vyacheslav Kungurtsev\footnote{Kungurtsev (vyacheslav.kungurtsev@fel.cvut.cz) is with the Department of Computer Science, Faculty of Electrical Engineering, Czech Technical University in Prague, 13 Karlovo Namesti, 12135 Prague, Czech Republic. His work was supported by
  the OP VVV project
CZ.02.1.01/0.0/0.0/16 019/0000765 ``Research Center for Informatics''}}

\begin{document}

\maketitle

\begin{abstract}
Langevin MCMC gradient optimization is a class of increasingly popular methods for estimating a posterior distribution.
This paper addresses the algorithm as applied in a decentralized setting, wherein data is distributed across
a network of agents which act to cooperatively solve the problem using peer-to-peer gossip communication.
We show, theoretically, results in 1) the time-complexity to $\epsilon$-consensus for the continuous
time stochastic differential equation, 2) convergence rate in $L^2$ norm to consensus for the discrete implementation
as defined by the Euler-Maruyama discretization and 3) convergence rate in the Wasserstein metric to the optimal stationary
distribution for the discretized dynamics.
\end{abstract}

\section{Introduction}
Consider the problem of sampling a posterior distribution $\pi$ on $\mathbb{R}^d$ with density
\[
\pi:x\to e^{-U(x)}/\int_{\mathbb{R}^d} e^{-U(y)}dy
\] 
with respect to the Lebesgue measure, with
a continuously differentiable potential $U:\mathbb{R}^d\to \mathbb{R}$. Now, consider
a decentralized computing environment. Specifically, we define the potential $U(x)=\sum_{i=1}^m u_i(x)$,
where we have a set of agents all of whom store local copies of estimates of $x_{(i)},\, i\in\{1,...,m\}$
and only have access to their respective $u_i(x)$. This can be an inherent property of the problem, with
the data defined on separate and distinct processes whose conglomerate minimization is encouraged. Alternatively,
the set up is methodological, where the use of a distributed computing platform is undertaken in order accelerate
the convergence towards the stationary distribution.

The communication network of the agent is modeled as a fixed undirected graph $\mathcal{G}\triangleq (\mathcal{V},\mathcal{E})$ with 
vertices $\mathcal{V}\triangleq \{1,..,I\}$ and $\mathcal{E}\triangleq\{(i,j)|i,j\in\mathcal{V}\}$ representing the agents and
communication links, respectively. We assume that the graph $\mathcal{G}$ is strongly connected. We note by $\mathcal{N}_i$ the
neighbors of $i$, i.e., $\mathcal{N}_i = \{j:(i,j)\in\mathcal{E}\}$.

We define the graph Laplacian matrix $\mathbf{L}=\mathbf{I}-\mathbf{W}$, where $\mathbf{W}=\mathbf{A}\otimes \mathbf{I}$ with
$\mathbf{A}$ satisfying $\mathbf{A}_{ij}\neq 0$ if $(i,j)\in\mathcal{E}$ and $\mathbf{A}_{ij}=0$ otherwise.

We assume that $\mathbf{W}$ is double stochastic (and symmetric, since the graph is undirected). The eigenvalues of $\mathbf{L}$ are real and can be sorted in a 
nonincreasing order $1=\lambda_1(\mathbf{L})> \lambda_2(\mathbf{L})\ge ...\ge \lambda_n(\mathbf{L})\ge 0$. 

Defining,
\[
\beta\triangleq \lambda_2(\mathbf{L})
\]
we shall make the following assumption,
\begin{assumption}\label{as:beta}
It holds that,
\[
\beta<1
\]
\end{assumption}
We shall define $\bar{\beta}$ to be the smallest eigenvalue of $\mathbf{L}$ that is nonzero.

Each agent $i$ has access to and controls an estimate of the primal stochastic variables $X^{(i)}(t)$. 

We make the following assumption about the potential function.
\begin{assumption}\label{as}
The potential $U(x)$ has a Lipschitz continuous gradient with constant $L$, i.e., $\|\nabla U(x)-\nabla U(y)\|\le L \|x-y\|$
and $U(\cdot)$ is strongly convex with constant $m$,
i.e., $U(x)-U(y)-\nabla U(y)(x-y)\ge \frac{m}{2}\|x-y\|^2$. Furthermore, the component functions
have bounded dissimilarity in gradients, i.e., $\|\nabla u_i(x)-\nabla u_j(y)\|\le G\|x-y\|$.
\end{assumption}
We consider convergence in terms of the Wasserstein distance $W_2$. For two measures $\mu$ and $\nu$ on $(\mathbb{R}^d,\mathcal{B}(\mathbb{R}^d))$
and for any real number $q\ge 1$, we define,
\[
W_q(\mu,\nu)=\left(\inf_{\rho\in\rho(\mu,\nu)} \int_{\mathbb{R}^d\times\mathbb{R}^d} \|x-y\|^q_2 d\rho(x,y)\right)^{1/q}
\]
where $\rho(\mu,\nu)$ is the set of joint distributions with $\mu$ and $\nu$ as marginals.

Consider the standard Langevin equation with Brownian motion $B_t$,
\[
dX_t = -\nabla U(X_t) dt+\sqrt{2\sigma } dB_t
\]
Convergence of this stochastic differential equation (SDE) and its discretization to the stationary distribution of the potential 
has been studied in a number of works, with increased interest in recent years due to the superiority of Langevin
gradient based approaches over sampling for log-concave potentials with high dimensional datasets. Consider now that there
are now $m$ different estimates $X^{(i)}$ for which we implement the following SDE,
\begin{equation}\label{eq:langevindist}
dX^{(i)}_t = -\sum_{j\in\mathcal{N}_i\cup\{i\}}\mathbf{L}_{ij}X^{(j)}(t) dt-\alpha(t)\nabla u_i(X^{(i)}(t)) dt+\sqrt{2\sigma\alpha(t)} dB^{(i)}_t
\end{equation}
and its Euler-Maruyana discretization with step-size $h$,
\begin{equation}\label{eq:langevindistdisc}
X^{(i)}_{k+1} = (1-h)X^{(i)}_{k}+h\sum_{j\in\mathcal{N}_i}\mathbf{W}_{ij}X^{(j)}_k -\alpha_{k+1}h \nabla u_i(X^{(i)}_k) dt+ \sqrt{2\sigma\alpha_{k+1}h} Z^{(i)}_{k+1}
\end{equation}

We assume a standard diminishing step-size,
\begin{equation}\label{eq:alphaform}
\alpha(t) = \frac{1}{1+t}
\end{equation}

If we consider the random vectors $X_t$ and $X_k$, respectively, as the stack, 
\[
X_t = \begin{pmatrix} (X^{(1)}_t)^T & (X^{(2)}_t)^T & ... & (X^{(m)}_t)^T \end{pmatrix}^T\text{ and }
X_k = \begin{pmatrix} (X^{(1)}_k)^T & (X^{(2)}_k)^T & ... & (X^{(m)}_k)^T \end{pmatrix}^T
\] 
and define $U_v(X) = \sum_{i=1}^m u_i(X^{(i)})$, then we can write the update of the full stack of vectors as,

\begin{equation}\label{eq:langevindistfull}
dX_t = -\mathbf{L}X(t) dt-\alpha(t)\nabla U_v(X(t)) dt+\sqrt{2\sigma\alpha(t)} dB_t
\end{equation}
\begin{equation}\label{eq:langevindistdiscfull}
X_{k+1} = \left((1-h)\mathbf{I}-h\mathbf{W}\right)X_k-\alpha_{k+1}h \nabla U_v(X_k) + \sqrt{2\sigma\alpha_{k+1}h} Z_{k+1}
\end{equation}

Consider the averaging operator $\frac{1}{m}\mathbf{1}\mathbf{1}^T\otimes I_d$. 
It is clear from the double stochasticity of $\mathbf{W}$ that 
$(\frac{1}{m}\mathbf{1}\mathbf{1}^T\otimes I_d)\mathbf{W} = (\frac{1}{m}\mathbf{1}\mathbf{1}^T\otimes I_d)$
and so $(\frac{1}{m}\mathbf{1}\mathbf{1}^T\otimes I_d)\mathbf{L}=0$.

Thus the average vectors $\bar{X}_t=(\frac{1}{m}\mathbf{1}\mathbf{1}^T\otimes I_d)X_t$ and 
$\bar{X}_k = (\frac{1}{m}\mathbf{1}\mathbf{1}^T\otimes I_d)X_k$ satisfy the updates,
\begin{equation}\label{eq:langevindistavg}
\begin{array}{l}
d\bar{X}_t = -\frac{\alpha(t)}{m}\sum_{i=1}^m \nabla u_i(X^{(i)}(t)) dt+\sqrt{2\sigma\alpha(t)} dB_t
\\ \qquad = -\alpha(t) \nabla U(\bar{X}(t))dt-\alpha(t)\left(\frac{1}{m}\sum_{i=1}^m \nabla u_i(X^{(i)}(t))-\nabla U(\bar{X}(t))\right)dt
+\sqrt{2\sigma\alpha(t)} d\bar{B}_t
\end{array}
\end{equation}
\begin{equation}\label{eq:langevindistdiscavg}
\begin{array}{l}
\bar X_{k+1} = \bar{X}_k-\frac{\alpha_{k+1}h}{m}\sum_{i=1}^m \nabla u_i(X^{(i)}_k) +\alpha_{k+1} \sqrt{2\sigma\alpha_{k+1}h} Z_{k+1}
\\ \qquad = \bar{X}_k-\alpha_{k+1} h\nabla U(\bar{X}_k)-\alpha_{k+1}h\left(\frac{1}{m}\sum_{i=1}^m \nabla u_i(X^{(i)}_k)-\nabla U(\bar{X}_k)\right)
+\sqrt{2\sigma\alpha_{k+1}h} \bar Z_{k+1}
\end{array}
\end{equation}
where we treat these terms in parentheses as bias error terms. Note that since a standard stochastic gradient
approximation has zero mean and bounded variance, just like the added normal noise, an extension to the
stochastic gradient Langevin case would be trivial and it is not included for ease of readability.

\subsection{Previous Work}
Langevin gradient methods have enjoyed a surge in popularity roughly since the publication of the popular article~\cite{welling2011bayesian}.
Other important papers concerning convergence and applications include~\cite{ma2015complete}
and~\cite{chen2016bridging}.

The paper~\cite{ahn2014distributed} considers stochastic gradient Langevin dynamics for distributed learning, without the 
network architecture considered in this work. The closest paper to this work is~\cite{swenson2019annealing}
which considers a stochastic gradient annealing algorithm for global optimization, see also the similar~\cite{vlaski2019distributed}.
The setting and algorithms are similar, however the focus and ultimate nature of the results is distinct in considering
asymptotic convergence as well as iteration convergence in expectation to a global minimizer, as opposed
to convergence to a desired stationary distribution for posterior sampling as considered here, with convergence defined
in appropriate distance of probability measures.

\section{Consensus}




To begin with, we consider a precisely defined appropriate notion of consensus in this setting, and show that
the SDE converges to consensus. We characterize the time-complexity to consensus at an arbitrary
desired $\epsilon$.
\begin{theorem}\label{th:cons}

Assume $\sigma<\bar\beta$. 

Let $\check{X}(t) = X(t)-1_m\otimes \bar{X}(t)$. 

Define $f(\check{X}_t)=\|\check{X}_t\|^2$. 

It holds that $\mathbb{E}[f(\check{X}_t)]$ converges to zero exponentially, i.e., 
we have that for any $\epsilon$ the time $T_\epsilon$ at which $\mathbb{E}[f(\check{X}_t)]\le \epsilon$ satisfies,
\[
\mathbb{E}[T_\epsilon]\le \frac{2}{\bar\beta-\sigma} \left[-\log\epsilon +\left(\mathbb{E}[f(\check{X}_0)]+\sigma \log \left(\frac{2L}{\bar\beta}\right)\right)e^{(\alpha(0)L-\bar\beta)(\frac{2L}{\bar\beta}-1)}+\frac{2L}{\bar\beta}-1\right]
\]
\end{theorem}
\begin{proof}
We have that,
\[
d\check{X}_t =  -\mathbf{L}\check{X}_tdt-\alpha(t)\left(\nabla U_v(X(t))-\mathbf{1}\otimes\frac{1}{m}\sum_{i=1}^m \nabla u_i(X^{(i)}(t))\right)dt
+\sqrt{2\sigma\alpha(t)} \left(dB_t - \mathbf{1}\otimes d\bar{B}_t\right)
\] 
We can write $d\check{B}_t = \left(dB_t - \mathbf{1}\otimes d\bar{B}_t\right)$ as itself a Wiener process since it is a
scaled sum of Wiener processes.

Now it holds that,
\[
\begin{array}{l}
\alpha(t)\left\|\nabla U_v(X(t))-\mathbf{1}\otimes\frac{1}{m}\sum_{i=1}^m \nabla u_i(X^{(i)}(t))\right\| \\ \qquad \le 
\alpha(t)\left\|\nabla U_v(X(t))-\mathbf{1}\otimes\nabla U(\bar{X}(t))+\mathbf{1}\otimes\nabla U(\bar{X}(t))-\mathbf{1}\otimes\frac{1}{m}\sum_{i=1}^m \nabla u_i(X^{(i)}(t))\right\| \\ \qquad\qquad\qquad \le 
2\alpha(t) L \|\check{X}(t)\|
\end{array}
\]

By It\^{o}'s Lemma, 
\[
df(\check{X}_t) = -2\check{X}_t^T \mathbf{L}\check{X}_tdt-2\alpha(t)\left(\nabla U_v(X(t))-\mathbf{1}\otimes\frac{1}{m}\sum_{i=1}^m \nabla u_i(X^{(i)}(t))\right)^T \check{X}_t dt
+\sigma\alpha(t) dt+\sqrt{2\sigma\alpha(t)} \check{X}_t^T d\check{B}_t.
\]
Note that $\check{X}_t$ is in the nullspace of $\mathbf{L}$ if and only if $\check{X}_t=0$. Thus, $\check{X}_t^T \mathbf{L}\check{X}_t\ge \bar{\beta}\|\check{X}_t\|^2$. And so we have,
\[
\begin{array}{l}
\mathbb{E}_t(f(\check{X}_t)) = f(\check{X}_0)-\mathbb{E}\left[\int_{s=0}^t \left\{2\check{X}_t^T \mathbf{L}\check{X}_t
+2\alpha(s)
\left(\nabla U_v(X(t))-\mathbf{1}\otimes\frac{1}{m}\sum_{i=1}^m \nabla u_i(X^{(i)}(t))\right)^T \check{X}_s 
-\frac 12\sigma\alpha(s) \right\}ds\right] \\
\qquad\le f(\check{X}_0)-2\mathbb{E}\left[\int_{s=0}^t\left\{\bar{\beta}f(\check{X}_s)- \alpha(s)L \|\check{X}_s\|^2 
-\frac 12\alpha(s)\sigma \right\}ds\right] \\ 
\qquad \le f(\check{X}_0)-2\mathbb{E}\left[\int_{s=0}^t \left\{(\bar\beta-\alpha(s)L) f(\check{X}_s) 
-\frac 12\alpha(s)\sigma \right\}ds\right] 
\end{array}
\]
Let $\hat T$ be such that $\alpha(\hat T)\le \frac{\bar\beta}{2L}$, i.e., $\hat T=\frac{2L}{\bar\beta}-1$. Then,
\[
\mathbb{E}_t(f(\check{X}_t))\le \mathbb{E}[f(\check{X}_0)]+2\mathbb{E}\left[\int_{s=0}^{\hat T}\left\{(\alpha(0)L-\frac 12\bar\beta) f(\check{X}_s) 
+\frac 12\alpha(s)\sigma \right\}ds\right] 
-2\mathbb{E}\left[\int_{s=\hat T}^t \left\{\frac{1}{2}\bar\beta f(\check{X}_s) -\frac 12\alpha(s)\sigma \right\}ds\right] 
\]
Recall the standard Grownwall's inequality,
\[
u(t)\le v+\int_a^t \beta(s) u(s) ds \Longrightarrow u(t)\le v \exp\left(\int_a^t \beta(s) ds\right)
\]
We get that,
\[
\mathbb{E}_t(f(\check{X}(\hat T))) \le \left(\mathbb{E}[f(\check{X}_0)]+\frac{\sigma}{2} \int_{s=0}^{\hat t} \alpha(s)ds\right)e^{(\alpha(0)L-\bar\beta)\hat T}
\le \left(\mathbb{E}[f(\check{X}_0)]+\sigma \log \left(\frac{2L}{\bar\beta}\right)\right)e^{(\alpha(0)L-\bar\beta)(\frac{2L}{\bar\beta}-1)}
\]

Now define $T_\epsilon$ as a stopping time at which it holds that $\mathbb{E}[f(\check{X}_s)]\le \epsilon$. Consider,
\[
\mathbb{E}(f(\check{X}_{T_\epsilon}))\le \mathbb{E}[f(\check{X}_{\hat T})]-\mathbb{E}\left[\int_{s=\hat T}^{T_\epsilon} \left\{\bar\beta f(\check{X}_s) -\alpha(s)\sigma \right\}ds\right] 
\]
Applying Grownwall's inequality again,
\[
\begin{array}{l}
\mathbb{E}(f(\check{X}_{T_\epsilon}))\le \left(\mathbb{E}[f(\check{X}_{\hat T})]+\sigma \int_{s=\hat T}^{\mathbb{E}[T_\epsilon]} \alpha(s) ds\right) e^{-\bar \beta(\mathbb{E}[T_\epsilon]-\hat T)} \\ \qquad\qquad \le  \left(\mathbb{E}[f(\check{X}_{\hat T})]+\sigma \log(1+\mathbb{E}[T_\epsilon])-\sigma\log(1+\hat T)\right) e^{-\bar \beta(\mathbb{E}[T_\epsilon]-\hat T)} \\ \qquad\qquad \le 
\left(\mathbb{E}[f(\check{X}_{\hat T})]+\sigma \mathbb{E}[T_\epsilon]\right) e^{-\bar \beta(\mathbb{E}[T_\epsilon]-\hat T)}
\end{array}
\]
which implies,
\[
\log(\epsilon) \le \log\left(\mathbb{E}[f(\check{X}_{\hat T})]+\sigma \mathbb{E}[T_\epsilon]\right)-\bar \beta(\mathbb{E}[T_\epsilon]-\hat T)
\le \mathbb{E}[f(\check{X}_{\hat T})]+\left(\sigma-\bar\beta\right)\mathbb{E}[T_\epsilon]+\frac{2L}{\bar\beta}-1
\]
and thus,
\[
\mathbb{E}[T_\epsilon] \le \frac{1}{\bar\beta-\sigma} \left[-\log\epsilon +\left(\mathbb{E}[f(\check{X}_0)]+\sigma \log \left(\frac{2L}{\bar\beta}\right)\right)e^{(\alpha(0)L-\bar\beta)(\frac{2L}{\bar\beta}-1)}+\frac{2L}{\bar\beta}-1\right]
\]

\end{proof}

Now consider the discrete result, i.e., a notion of convergence rate to consensus for the
stochastic process defined by the Euler-Maruyama discretization $X_k$. 

We will need to use a classic convergence Lemma
\begin{lemma}\cite[Lemma 2.5]{polyak1987introduction}\label{lem:polyakone}
Let $u_k\ge 0$ and,
\[
u_{k+1}\le \left(1-\frac{c}{k^s}\right) u_k+\frac{d}{k^{t}}
\]
with $0<s<1$ and $s<t$ and $c,d>0$. Then,
\[
u_k\le \frac{d}{c}\frac{1}{k^{t-s}}+o\left(\frac{1}{k^{t-s}}\right)
\]
\end{lemma}

\begin{theorem}
The distribution associated with the consensus error $\check{X}_{k} := X_{k}-1\otimes \bar{X}_{k} $ converges in $W_2(\cdot,\cdot)$ distance, for any $\gamma>0$,
to the Dirac delta at zero at a rate of
\[
O\left(\frac{1}{(1+k)^{1/2-\gamma}}\right)
\]
\end{theorem}
\begin{proof}
We have that 
\[
\begin{array}{l}
\check{X}_{k+1} := X_{k+1}-1\otimes \bar{X}_{k+1} \\ \qquad = 
\left((1-h)\mathbf{I}-h\mathbf{W}\right)X_k-\mathbf{1}\otimes\bar{X}_{k}-\alpha_{k+1} h\nabla U_v(X_k) + \sqrt{2\sigma\alpha_{k+1}h} Z_{k+1} \\ \qquad\qquad\qquad\qquad+ 
\frac{\alpha_{k+1}h}{m}\sum_{i=1}^m \mathbf{1}\otimes \nabla u_i(X^{(i)}_k) -\sqrt{2\sigma\alpha_{k+1} h} 1\otimes \bar Z_{k+1} \\
\qquad = \check{X}_k-h\mathbf{L}X_k-\alpha_{k+1}h \left(\nabla U_v(X_k)-\frac{1}{m}\sum_{i=1}^m \mathbf{1}\otimes \nabla u_i(X^{(i)}_k)\right)+\sqrt{2\sigma\alpha_{k+1}h }\check{Z}_{k+1} \\
\qquad = (1-h)\mathbf{L}\check{X}_k-\alpha_{k+1}h \left(\nabla U_v(X_k)-\frac{1}{m}\sum_{i=1}^m \mathbf{1}\otimes \nabla u_i(X^{(i)}_k)\right)+\sqrt{2\sigma\alpha_{k+1}h }\check{Z}_{k+1}
\end{array}
\]

Let $\nu_k$ be the distribution associated with the stochastic process $\check{X}_k$.
Define the distribution for consensus of $\check{X}$ to be $\pi$, the delta function at zero, i.e., $\pi = \delta_0(x)$. 
Define the distribution $\pi_k$ to be the normal distribution with standard deviation $2\sigma h/\sqrt{1+k}$. 
  
Construct now a stochastic variable $Y_0\sim \pi_0$ such that the Wasserstein distance to the initial distribution is
minimized, i.e., $W_2(\nu_0,\pi_0)=\|Y_0-\check{X}_{0}\|$ and $Y_{k+1}=\sqrt{2\sigma\alpha_{k+1}h}\check{Z}_{k+1}$.
Note that this is a process that for each $k$ has $\pi_k$ as its associated distribution. 
Since the support is $\mathbb{R}^d$, each $Y_k$ is in the support of $\pi_k$. 
Since zero, the only vector in the support of $\pi$, is also
in the support of $\pi_k$, it holds that $W_2(\pi,\pi_k)=\sqrt{\int_{\mathbb{R}^d} \|x\|^2 d\pi_k}=2\sigma h/\sqrt{k+1}$ and we can finally write,

\begin{equation}\label{eq:wasdistcons}
W_2(\nu_k,\pi) \le W_2(\nu_k,\pi_k)+W_2(\pi_k,\pi)\le \mathbb{E}[\|\check{X}_k-Y_k\|_{L^2}]+2\sigma h/\sqrt{k+1}
\end{equation}

Consider now $\check{X}_{k+1}-Y_{k+1}$,
\[
\begin{array}{l}
\check{X}_{k+1}-Y_{k+1} = (1-h)\mathbf{L}\check{X}_k-\alpha_{k+1}h \left(\nabla U_v(X_k)-\frac{1}{m}\sum_{i=1}^m \mathbf{1}\otimes \nabla u_i(X^{(i)}_k)\right)
\end{array}
\]
and thus, using Assumption~\ref{as} we can see that,
\[
\|\check{X}_{k+1}-Y_{k+1}\| = (1-h-2\alpha_{k+1}hmG)\|\check{X}_k-Y_k\| +\|Y_k\|
\]
and so,
\[
\mathbb{E}\|\check{X}_{k+1}-Y_{k+1}\| = (1-h-2\alpha_{k+1}hmG)\mathbb{E}\|\check{X}_k-Y_k\| +\mathbb{E}\|Y_k\|
\]

Finally we can apply Lemma~\ref{lem:polyakone} to conclude that for any $\gamma>0$, for sufficiently large $k\ge K$,
\[
\mathbb{E}\|\check{X}_{k+1}-Y_{k+1}\| = \left(1-\frac{2h}{(k+1)^{\gamma}}\right)\mathbb{E}\|\check{X}_k-Y_k\| +\frac{2\sigma h}{\sqrt{k+1}} = O\left(\frac{1}{k^{1/2-\gamma}}\right)
\]
Plugging this into~\eqref{eq:wasdistcons} yields the final result.


\end{proof}

\section{Convergence}

Finally we show that the average process $\bar{X}_k$ converges to a minimizer of $U(\cdot)$.

Let $S(t)=\int_{0}^t \alpha(s) ds$ and let $T:\mathbb{R}\to\mathbb{R}$ be the inverse of $S$ such that
$S(T(t))=t$ (which exists since $T$ is increasing. Letting $Y(t) = \bar{X}(T(t))$, we have
$\frac{d}{dt}Y(t)=\frac{d\bar{X}(T(t))}{dt}\frac{dT(t)}{dt}$ and $\frac{dS(T(t))}{dt}\frac{dT(t)}{dt} = 1$
so $\frac{dT(t)}{dt} = \frac{1}{\alpha(T(t))}$ (see~\cite{swenson2019distributed}). Now the process $Y_t$
satisfies,

\[
dY_t = -h\nabla U(\bar{X}(T(t)))dt-h\left(\frac{1}{m}\sum_{i=1}^m \nabla u_i(X^{(i)}(T(t)))-\nabla U(\bar{X}(T(t)))\right)dt
+\frac{\sqrt{2\sigma\alpha(T(t))h}}{\alpha(T(t))} d\bar{B}(T(t))
\]
and by the scale invariance of a Wiener process ($\beta^{-1} W_{\beta^2 t}=W_t$ for all $\beta>0$) this is equivalent to,
\[
dY_t = -h\nabla U(Y_t)dt-h\left(\frac{1}{m}\sum_{i=1}^m \nabla u_i(X^{(i)}(T(t)))-\nabla U(\bar{X}(T(t)))\right)dt
+\sqrt{2\sigma h} d\bar{B}_t
\]
with discretization,
\[
Y_{k+1} = Y_k-h\nabla U(Y_k)-h\left(\frac{1}{m}\sum_{i=1}^m \nabla u_i(Y^{(i)}_k)-\nabla U(\bar{Y}_k)\right)
+\sqrt{2\sigma h} \bar{Z}_{k+1}
\]
We observe that with $\alpha_k=\frac{1}{1+k}$, it holds that $S_k=\log(1+k)$ and so $T(t_k) = e^{t_k}-1$. Thus
$\alpha_k = \frac{1}{e^k}\le\frac{1}{1+k}$.

Let us redefine $\nu_k$ to be the distribution associated with the stochastic process $Y_k$ and $\pi$ the stationary process
associated with $U(x)$. Consider that in general now we have iteration dependant stepsize $h_k$.

To derive our diminishing step-size convergence result, we recall a useful Lemma.
\begin{lemma}\cite[Lemma 2.4]{polyak1987introduction}\label{lem:polyak}
Let $u_k\ge 0$ and,
\[
u_{k+1}\le \left(1-\frac{c}{k}\right) u_k+\frac{d}{k^{p+1}}
\]
with $d>0$, $p>0$ and $c>0$ and $c>p$. Then,
\[
u_k\le d(c-p)^{-1}k^{-p}+o(k^{-p})
\]
\end{lemma}

We are now ready to prove the main convergence result.

\begin{theorem}
If $h_k$ is constant, i.e., $h_k=h$ and $h<\min\left\{\frac{1}{L},m\right\}$ then,
\[
\lim_{K\to\infty} W_2(\nu_{K},\pi)\le \frac{\chi L (hd)^{1/2}}{m}
\]
with $\chi=7\sqrt{2}/6$. If $h_k=\frac{1}{k}$ then for $k\ge \frac{L+m}{2}-1$,
\[
W_2(\nu_{k},\pi) = O\left(k^{-1/2}\right)
\]
\end{theorem}
\begin{proof}
We apply~\cite[Proposition 2]{dalalyan2019user} to $\nu_k$ with the deterministic bias bound $\|\zeta_k\|\le L\alpha_k$ to get,
\begin{equation}\label{eq:wrecursion}
W_2(\nu_{k+1},\pi) \le \rho_{k+1} W_2(\nu_k,\pi)+\chi L(h^3_{k+1} d)^{1/2}+Lh_{k+1}\alpha_{k+1} 
\end{equation}
where $\rho_k = \max(1-mh_{k+1},Lh_{k+1}-1)$.

If $h_k=h$ such that $\rho_k<1$ (i.e., $h<\min\left\{\frac{1}{L},m\right\}$) then we have,
\[
W_2(\nu_{K},\pi)\le \rho^K W(\nu_0,\pi)+\sum_{k=0}^K \rho^{K-k} \chi L(h^3 d)^{1/2}+\sum_{k=0}^K \frac{Lh\rho^{K-k}}{1+k}
\]
and we use~\cite[Lemma 7a]{di2016next} to conclude that the last term approaches zero.

Now assume that $h_k=\frac{1}{k}$, and let $\hat K$ be the first iteration from which $1-\frac{m}{k+1}\ge \frac{L}{k+1}-1$, or $\frac{L+m}{k+1}\le 2$,
i.e., $\hat K=\frac{L+m}{2}-1$. We have for $k\ge \hat K$,
\[
\begin{array}{l}
W_2(\nu_{k+1},\pi) \le \left(1-\frac{m}{k+1}\right) W_2(\nu_k,\pi)+\frac{\chi L d^{1/2}}{(k+1)^{3/2}}+\frac{L}{(k+1)^2} 
\\ \qquad\qquad \qquad \le \left(1-\frac{m}{2k}\right) W_2(\nu_k,\pi)+\frac{\chi L d^{1/2}+L}{k^{3/2}} 
\end{array}
\]

Applying Lemma~\ref{lem:polyak} we obtain that,
\[
W_2(\nu_k,\pi) \le \frac{\chi L d^{1/2}+L}{\left(\frac{m}{2}-\frac{1}{2}\right) k^{1/2}}+o(k^{-1/2})
\]

\end{proof}
\section{Conclusion}
In this paper we derived convergence rate results in appropriate notions of probability measure distance for the
stochastic gradient Langevin dynamics method in a decentralized setting. These results confirm that the performance
of this powerful method for obtaining the stationary distribution associated with log-concave potentials
extends to a distributed network communication setting. Given the promising theoretical results
we are aiming to perform an extensive set of numerical experiments.

\bibliographystyle{plain}
\bibliography{refs}

\end{document}